\documentclass[11pt]{article}
\usepackage[english]{babel}
\usepackage{amsfonts}
\usepackage[utf8]{inputenc}
\usepackage{amsthm}
\usepackage{mathrsfs}
\usepackage{amsmath}
\usepackage{enumerate}
\usepackage{cite}
\usepackage{bbm}
\usepackage{makeidx}
\usepackage{xcolor}
\usepackage{graphicx}
\usepackage{frontespizio}
\usepackage[makeroom]{cancel}
\usepackage[normalem]{ulem}
\usepackage{amssymb}
\usepackage[a4paper, total={7in, 8in}]{geometry}
\usepackage{a4wide}
\numberwithin{equation}{section}
\usepackage[colorlinks,citecolor=green,linkcolor=red]{hyperref}
\usepackage{mathtools}
\usepackage{enumitem}
\usepackage[toc,page]{appendix}
\usepackage{esint}

\newcommand{\nchi}{{\raise.3ex\hbox{$\chi$}}}

\newcommand{\sfd}{{\sf d}}

\renewcommand{\phi}{\varphi}
\newcommand{\restr}[1]{\lower3pt\hbox{$|_{#1}$}}

\newcommand{\X}{{\rm X}}

\newcommand{\fr}{\penalty-20\null\hfill$\blacksquare$} 
\definecolor{mygray}{gray}{0.9}

\newcommand{\mea}{\mathfrak{m}}
\newcommand{\mm}{\mathfrak{m}}

\newcommand{\LIP}{{\rm LIP}}

\newcommand{\lims}{\varlimsup}

\renewcommand{\d}{{\mathrm d}}

\newcommand{\lip}{{\rm lip}}

\newcommand{\RCD}{\mathsf{RCD}}

\newcommand{\R}{\mathbb{R}}

\newcommand{\N}{\mathbb{N}}

\theoremstyle{plain}
\newtheorem{theorem}{Theorem}[section]
\newtheorem{lemma}[theorem]{Lemma}

\newtheorem{proposition}[theorem]{Proposition}
\newtheorem{cor}[theorem]{Corollary}

\newtheorem{question}[theorem]{Question}

\theoremstyle{definition}
\newtheorem{definition}[theorem]{Definition}
\newtheorem{remark}[theorem]{Remark}

\newcommand{\ezero}{E^{(0)}}
\newcommand{\eone}{E^{(1)}}
\newcommand{\Int}{{\rm int}}
\newcommand{\res}{\mathop{\hbox{\vrule height 7pt width .5pt depth 0pt
\vrule height .5pt width 6pt depth 0pt}}\nolimits} %restriction of measures

\title{Topological regularity of isoperimetric sets \\ in PI spaces having a deformation property}
\author{Gioacchino Antonelli\footnote{\href{mailto:ga2434@nyu.edu}{ga2434@nyu.edu}, Courant Institute of Mathematical Sciences (NYU), 251 Mercer Street, 10012, New York, USA.},
\,Enrico Pasqualetto\footnote{\href{mailto:enrico.e.pasqualetto@jyu.fi}{enrico.e.pasqualetto@jyu.fi}, Department of Mathematics and Statistics, P.O.\ Box 35 (MaD), FI-40014 University of Jyvaskyla, Finland.},
\,Marco Pozzetta\footnote{\href{mailto:marco.pozzetta@unina.it}{marco.pozzetta@unina.it}, Dipartimento di Matematica e Applicazioni, Universit\`{a} di Napoli Federico II, Via Cintia, Monte S.\ Angelo, 80126 Napoli, Italy.},
\,and Ivan Yuri Violo\footnote{\href{mailto:ivan.y.violo@jyu.fi}{ivan.y.violo@jyu.fi}, Department of Mathematics and Statistics, P.O.\ Box 35 (MaD), FI-40014 University of Jyvaskyla, Finland.}}
\begin{document}
\maketitle
\begin{abstract}
We prove topological regularity results for isoperimetric sets in PI spaces having a suitable deformation property, which
prescribes a control on the increment of the perimeter of sets under perturbations with balls. More precisely, we prove that
isoperimetric sets are open, satisfy boundary density estimates and, under a uniform lower bound on the volumes of unit balls,
are bounded. Our results apply, in particular, to the class of possibly collapsed $\mathrm{RCD}(K,N)$ spaces.
As a consequence, the rigidity in the isoperimetric inequality on possibly collapsed $\mathrm{RCD}(0,N)$ spaces with Euclidean
volume growth holds without the additional assumption on the boundedness of isoperimetric sets.
Our strategy is of interest even in  the Euclidean setting, as it simplifies some classical arguments.
\end{abstract}
\medskip
\noindent\textbf{MSC(2020).} Primary: 53C23, 49Q20. Secondary: 26B30, 26A45, 49J40.\\
\textbf{Keywords.} Isoperimetric set, PI space, deformation property.
\section{Introduction}

In this paper we consider length PI spaces, i.e.\ metric measure spaces $(\X,\sfd,\mm)$ where $\mm$ is a uniformly locally doubling
Borel measure, there holds a weak local $(1,1)$-Poincaré inequality (see Definition \ref{def:PI}), and the distance between any two
points $x,y$ is realized as the infimum of the lengths of curves joining $x$ and $y$. The well-established theory of $BV$ functions
on metric measure spaces \cite{amb01, MIRANDA2003} allows the treatment of sets of finite perimeter in this generalized setting.
Hence, it makes sense to consider the classical isoperimetric problem, defined by the following minimization:
\[
\inf\big\{P(E)\;\big|\;E\subseteq\X\text{ Borel, }\mm(E)=v\big\},
\]
for any assigned volume $v\in(0,\mm(\X))$, where $P(E)$ denotes the perimeter of $E$.
A set $E$ minimizing the previous infimum is called an \emph{isoperimetric set}, or an \emph{isoperimetric region}.

One of the fundamental questions about isoperimetric sets addresses their topological regularity. Namely, one aims at proving that, up to the choice of a representative,
isoperimetric sets are open, bounded and  enjoy density estimates at points of
the topological boundary. In the Euclidean space, topological regularity follows from \cite{GonzalezMassariTamanini}, subsequently
generalized in \cite{Xia05}. The proof in the Euclidean setting can be further simplified, see \cite[Example 21.3, Theorem 21.11]{MaggiBook}. On Riemannian manifolds the result is due to \cite{morgan2003regularity}.
In \cite{AntonelliPasqualettoPozzetta21} the result has been generalized to the setting of \emph{noncollapsed} $\RCD(K,N)$ spaces
$(\X,\sfd,\mathcal{H}^N)$, i.e.\ \(N\in\N\) and \(\mm\) coincides with the Hausdorff measure \(\mathcal H^N\).
We mention also \cite{KKLS13}, which addresses the case of quasi-minimal sets in PI spaces.
\medskip

The purpose of this paper is to prove the topological regularity of isoperimetric sets in the general setting of length PI spaces
that enjoy a so-called \emph{deformation property}, which we are going to introduce (we refer to Definition \ref{def:deform} for
the precise definition). We say that a metric measure space \((\X,\sfd,\mm)\) has the deformation property provided the following
holds: given a set \(E\subseteq\X\) of finite perimeter and a point \(x\in\X\), we can find \(R,C>0\) such that
\begin{equation}\label{eq:deformINTRO}
P(E\cup B_r(y))\leq C\frac{\mm(B_r(y)\setminus E)}{r}+P(E)\quad\text{ for every }y\in B_R(x)\text{ and }r\in(0,R).
\end{equation}
Classes of spaces having the deformation property are collected in Remark \ref{rmk:ex_deform_prop}. Notably, the class includes $\RCD(K,N)$ spaces $(\X,\sfd,\mm)$, thanks to \cite[Theorem 1.1]{AntonelliPasqualettoPozzetta21}. We shall not introduce
$\RCD$ spaces here, and we refer the reader to the survey \cite{AmbICM} and to the references therein.

We point out that being a PI space does not imply that the deformation property holds, see the examples in Remark \ref{rmk:example_PI_no_def} and in Remark \ref{rem:Exaple2PINoDef}. Anyway, we are not aware of any example of a PI space where the deformation
property fails when tested on an isoperimetric set $E$, nor of an example of a PI space where the essential interior an some isoperimetric set is not topologically open.
\medskip

Deformation properties for sets of locally finite perimeter are well-known in the smooth context \cite[Lemma 17.21]{MaggiBook},
and they represent a tool of crucial importance in several classical arguments. We refer, for instance, to
\cite[VI.2(3)]{AlmgrenBook}, to \cite[Lemma 4.5]{GalliRitore} and \cite[Lemma 3.6]{Pozuelo} in the sub-Riemannian setting,
and to \cite{CintiPratelli, PratelliSaracco} which study isoperimetric problems in a weighted setting.

In fact, it is mostly powerful to couple the topological regularity of an isoperimetric set, or of a set minimizing some
variational problem, with the deformation property. For instance, knowing that such a set $E$ has an open representative
allows to apply \eqref{eq:deformINTRO} centered at points $y$ in the interior, so that $\mm(E\cup B_r(y))>\mm(E)$ only
for radii $r$ sufficiently large, and thus \eqref{eq:deformINTRO} implies that one can increase the volume of $E$ controlling
the perimeter of the deformed set $E\cup B_r(y)$ \emph{linearly} with respect to the increase of mass $\mm(B_r(y)\setminus E)$.
An analogous observation holds applying \eqref{eq:deformINTRO} to the complement, in case the complement of the considered set
has an open representative. Observe that the previous improved deformation property with linear control follows from \eqref{eq:deformINTRO} only \emph{after} topological regularity of the set has been established. This is in contrast with the Euclidean setting, where the stronger form of deformation property is always available \cite[Lemma 17.21]{MaggiBook}. The latter result follows by deforming sets of finite perimeter by flows of vector fields, an argument out of reach in the metric setting. Hence the simplest Euclidean proof for the topological regularity of isoperimetric sets \cite[Example 21.3]{MaggiBook} has no hope of being performed in our framework, and we must look for an alternative argument.

\medskip

We can now state our main result, which yields the topological regularity at the more general level of \emph{volume-constrained
minimizers} of the perimeter, i.e.\ sets which minimize the perimeter with respect to any bounded variation that locally preserves
the measure, see Definition \ref{def:VolumeConstrained}.
We will denote by \(E^{(1)}\), \(E^{(0)}\), and \(\partial^e E\) the essential interior, the essential exterior, and the essential
boundary, respectively, of a Borel set \(E\subseteq\X\); see Section \ref{s:PI} for their definitions.
\begin{theorem}[Topological regularity of volume-constrained minimizers]\label{thm:main}
Let \((\X,\sfd,\mm)\) be a length PI space having the deformation property. Let \(E\subseteq\X\) be a volume-constrained minimizer
of the perimeter. Then \(E^{(1)}={\rm int}(E^{(1)})\) and \(E^{(0)}={\rm int}(E^{(0)})\). In particular, it holds that \(E^{(1)}\),
\(E^{(0)}\) are open sets and \(\partial E^{(1)}=\partial E^{(0)}=\partial^e E\).
\end{theorem}

The previous theorem implies density estimates on the volume and on the perimeter measure of a volume-constrained minimizer at
points of the essential boundary, see Theorem \ref{thm:density}. For an isoperimetric set, we can additionally prove its boundedness. Namely:
\begin{theorem}\label{thm:isoper_bdd}
Let \((\X,\sfd,\mm)\) be a length PI space having the deformation property. Suppose that $\inf_{x \in \X} \mea(B_1(x))>0$.
Let \(E\subseteq\X\) be an isoperimetric set. Then \(E^{(1)}\) is bounded. In particular, every isoperimetric set in \(\X\)
has a bounded representative.
\end{theorem}

Since $\RCD(K,N)$ spaces with \(N<\infty\) are length PI spaces (see \cite{Rajala12,Sturm06II} and \cite{Lott-Villani09}),
and as recalled above they have the deformation property, putting together Theorem \ref{thm:main} and Theorem \ref{thm:isoper_bdd}
we obtain the following.
\begin{cor}\label{cor:BoundedCollapsed}
Let $(\X,\sfd,\mm)$ be an $\RCD(K,N)$ space with $N<\infty$. Let $E\subseteq\X$ be an isoperimetric set.
Then the sets $E^{(1)}, E^{(0)}$ are open and $\partial^e E=\partial E^{(1)} = \partial E^{(0)}$.
Moreover, if in addition $\inf_{x\in\X}\mm(B_1(x))>0$, then $E^{(1)}$ is bounded.
\end{cor}

In the case of noncollapsed $\RCD(K,N)$ spaces,
the above result has been previously proved in \cite[Theorem 1.4]{AntonelliPasqualettoPozzetta21}.
\medskip

As an application of Corollary \ref{cor:BoundedCollapsed}, we can refine the rigidity part in the sharp isoperimetric inequality
on $\RCD(0,N)$ spaces $(\X,\sfd, \mm)$ with Euclidean volume growth. We recall that ``Euclidean volume growth'' means that the
\emph{asymptotic volume ratio}
\[
{\rm AVR}(\X,\sfd,\mm)\coloneqq\lim_{R\to\infty}\frac{\mm(B_R(p))}{\omega_N R^N},\quad\text{ for some }p\in\X,
\]
of the space is strictly positive. Recall that the existence of the above limit is guaranteed by the monotonicity of
\((0,+\infty)\ni r\mapsto\mm(B_r(p))/\omega_N r^N\), which in turn follows from the Bishop--Gromov inequality
(see e.g.\ \cite{Sturm06II}). Observe that the condition \({\rm AVR}(\X,\sfd,\mm)>0\) implies that \(\inf_{x\in\X}\mm(B_1(x))>0\).
The sharp isoperimetric inequality on these spaces, see \eqref{eq:Levy-Gromov} below, was obtained at different levels of 
generality in \cite{AgostinianiFogagnoloMazzieri, Brendle21, BaloghKristaly, APPSb, CavallettiManini, CavMan22}.
In \cite{APPSb} the rigidity for the isoperimetric inequality was proved for noncollapsed $\RCD(0,N)$ spaces.
In the recent \cite[Theorem 1.5]{CavMan22}, the authors prove the rigidity for the inequality in all $\RCD(0,N)$ spaces with Euclidean
volume growth under the additional assumption that the set achieving the equality is bounded.
An application of our Corollary \ref{cor:BoundedCollapsed} allows to drop the previous boundedness requirement,
thus obtaining the following unconditional rigidity statement.
\begin{theorem}[Sharp and rigid isoperimetric inequality on \({\sf RCD}(0,N)\) spaces with Euclidean volume growth]\label{thm:MAINMAIN}
Let \((\X,\sfd,\mm)\) be an \({\sf RCD}(0,N)\) space with \(1<N<\infty\) having Euclidean volume growth.
Then for every set of finite perimeter \(E\subseteq\X\) with \(\mm(E)<+\infty\) it holds that
\begin{equation}\label{eq:Levy-Gromov}
P(E)\geq N\omega_N^{\frac{1}{N}}{\rm AVR}(\X,\sfd,\mm)^{\frac{1}{N}}\mm(E)^{\frac{N-1}{N}}.
\end{equation}
Moreover, the equality in \eqref{eq:Levy-Gromov} holds for some set of finite perimeter \(E\subseteq\X\) with \(\mm(E)\in(0,+\infty)\)
if and only if \(\X\) is isometric to a Euclidean metric measure cone over an \({\sf RCD}(N-2,N-1)\) space and \(E\) is isometric,
up to negligible sets, to a ball centered at one of the tips of \(\X\).
\end{theorem}
In the previous theorem, when we say that \textit{$\X$ is a Euclidean metric measure cone over an \({\sf RCD}(N-2,N-1)\) space} we mean that there is a compact \({\sf RCD}(N-2,N-1)\) metric measure space $(Z,\sfd_Z,\mm_Z)$ such that $(\X,\sfd,\mm)$ is isomorphic, as a metric measure space, to the metric measure cone $(C(Z),\sfd_{c},t^{N-1}\mathrm{d}t\otimes\mathrm\mm_Z)$, where $\sfd_c$ is the cone metric built using $\sfd_Z$. In case $1<N<2$, it is understood that in the rigidity part of the previous statement, the space $X$ is either a weighted Euclidean half-line or a weighted Euclidean line.\\
We stress that Theorem \ref{thm:MAINMAIN} is not a straightforward consequence of the results in \cite{APPSb}, according to which the same result holds in the class of noncollapsed spaces. Indeed, an $\RCD(0,N)$ space with $1<N<\infty$ and with Euclidean volume growth might not be noncollapsed. A simple example is given by the weighted Euclidean half-line $([0,+\infty),\sfd_{\mathrm{eu}},t^{N-1}\mathrm{d}t)$, with $N>1$.

\medskip
We now briefly discuss our strategy for the proof of Theorem \ref{thm:main}. As mentioned above, the Euclidean proof \cite[Example 21.3]{MaggiBook} cannot be adapted to our setting.
As in the classical \cite{GonzalezMassariTamanini,Xia05}, we gain information on a volume-constrained minimizer
by comparison with suitable competitors exploiting the deformation property, but our argument is different, more direct, and much shorter. The strategy of
\cite{GonzalezMassariTamanini,Xia05} consists in proving first that $E$ has an interior and  an exterior point, i.e.\ $\Int(\eone)\neq \varnothing$ and $ \Int(\ezero)\neq \varnothing$ (see
\cite[Theorem 1]{GonzalezMassariTamanini}), then one deduces that $E$ is a $(\Lambda,r_0)$-perimeter minimizer, and thus finally that $E$ is open. Instead, we prove directly that if
$x \in \ezero$ and $y \in \eone$ are arbitrary points, then $x\in \Int(\ezero)$ and $y \in  \Int(\eone)$. To do so we avoid deriving quantitative estimates on
the decay of $\mm(B_r(y)\setminus E)$ as in \cite{GonzalezMassariTamanini, Xia05, AntonelliPasqualettoPozzetta21}, and we rather adopt a more qualitative approach.
More precisely, the key point is to show (see the key Lemma \ref{lem:key_lemma}) that if the function $v(r)\coloneqq \mm(B_r(x)\cap E)$  vanishes, as $r\to 0^+$,
slower than the function $w(r)\coloneqq \mm(B_r(y)\setminus E)$, then $x \in \Int(\ezero)$ (and vice versa for $y \in \Int(\eone)$).
By ``slower'' we mean, roughly speaking, that $v(r)\ge w(r)$ for many $r>0$ in a measure-theoretic sense
(see Lemma \ref{lem:key_lemma} for the precise statement). However, up to exchanging $E$ with its complement $\X\setminus E$,
we can always ensure that $v(r)$ vanishes slower than $w(r)$, thus deducing that $x \in \Int(\ezero)$. By symmetry, we get $y \in \Int(\eone)$ as well.
\medskip

We point out that the strategy of \cite{GonzalezMassariTamanini,Xia05} does not seem to
generalize to our setting, unless we require additional assumptions -- such as Ahlfors regularity -- which we do not want
to make (in order to obtain a result which applies to the whole class of collapsed \({\sf RCD}(K,N)\) spaces). This motivated
us to look for an alternative proof of the topological result, which -- we believe -- is of interest even in the Euclidean setting,  since it brings simplifications to the classical arguments in \cite{GonzalezMassariTamanini}, still (necessarily) avoiding the use of the smooth structure of the ambient.
\medskip

{
We conclude the introduction by explicitly recording the following open problem.

\begin{question}
    Let \((\X,\sfd,\mm)\) be a length PI space and let \(E\subseteq\X\) be a volume-constrained minimizer
of the perimeter. Is it true that $E^{(1)}$ is open?
\end{question}
}

\textbf{Acknowledgements.}
Part of this research has been carried out at the Fields Institute (Toronto) in November 2022, during the Thematic Program on
Nonsmooth Riemannian and Lorentzian Geometry. The authors gratefully acknowledge the warm hospitality and the stimulating atmosphere.
The authors thank Panu Lahti for pointing out the example in Remark \ref{rem:Exaple2PINoDef}.
The authors also thank Camillo Brena, Vesa Julin, Tapio Rajala, and Daniele Semola for fruitful discussions on the topic of the paper. The authors also thank the Reviewer for the careful reading and for pointing out an inaccuracy in a preliminary version of the paper.
\section{Preliminaries}
Given a metric space \((\X,\sfd)\), we denote by \(\LIP_{\mathrm{loc}}(\X)\) the space of all locally Lipschitz functions from \(\X\) to \(\R\),
{i.e.\ of those functions \(f\colon\X\to\R\) such that for any \(x\in\X\) there exists \(r_x>0\) for which \(f\) is Lipschitz on \(B_{r_x}(x)\).}
The \emph{slope} \(\lip(f)\colon\X\to[0,+\infty)\) of a function \(f\in\LIP_{\mathrm{loc}}(\X)\) is defined as \(\lip(f)(x)\coloneqq 0\) if \(x\in\X\) is an isolated point and
\[
\lip(f)(x)\coloneqq\lims_{y\to x}\frac{|f(x)-f(y)|}{\sfd(x,y)}\quad\text{ if }x\in\X\text{ is an accumulation point.}
\]
The topological interior and the topological boundary of a set \(E\subseteq\X\) are denoted by \({\rm int}(E)\) and \(\partial E\), respectively.
{
A Borel measure \(\mu\geq 0\) on \(\X\) is \emph{locally finite} if for any \(x\in\X\) there exists \(r_x>0\) such that \(\mu(B_{r_x}(x))<+\infty\),
while we say that \(\mu\) is \emph{boundedly finite} if \(\mu(B)<+\infty\) whenever \(B\subseteq\X\) is bounded Borel. Trivially, each boundedly finite measure
is locally finite, while the converse holds e.g.\ if \((\X,\sfd)\) is \emph{proper}, i.e.\ bounded closed subsets of \(\X\) are compact.
Notice that locally finite Borel measures on a complete separable metric space are \(\sigma\)-finite.
}
\subsection{Sets of finite perimeter in metric measure spaces}
In this paper, by a \emph{metric measure space} \((\X,\sfd,\mm)\) we mean a complete separable metric space \((\X,\sfd)\) together with a boundedly finite Borel measure \(\mm\geq 0\) on \(\X\).
Following \cite{MIRANDA2003}, we define the \emph{total variation} \(|{\bf D}f|(B)\in[0,+\infty]\) of a given function \(f\in L^1_{\mathrm{loc}}(\X)\) in a Borel set \(B\subseteq\X\) as
\[
|{\bf D}f|(B)\coloneqq\inf_{B\subseteq\Omega\text{ open}}\inf\bigg\{\varliminf_{n\to\infty}\int_\Omega\lip(f_n)\,\d\mm\;\bigg|\;
(f_n)_{n\in\N}\subseteq\LIP_{\mathrm{loc}}(\Omega),\,f_n\to f\text{ in }L^1_{\mathrm{loc}}(\Omega)\bigg\}.
\]
If for some open cover \((\Omega_n)_{n\in\N}\) of \(\X\) we have that \(|{\bf D}f|(\Omega_n)<+\infty\) holds for every \(n\in\N\), then \(|{\bf D}f|\) is a locally finite Borel measure on \(\X\).
We say that a Borel set \(E\subseteq\X\) is \emph{of locally finite perimeter} if \(P(E,\cdot)\coloneqq|{\bf D}\nchi_E|\) is a locally finite measure, called the \emph{perimeter measure}
of \(E\). When \(P(E)\coloneqq P(E,\X)<+\infty\), we say that \(E\) is \emph{of finite perimeter}.
\begin{remark}\label{rmk:bdry_null}{\rm
If \(E\subseteq\X\) is a set of locally finite perimeter and \(x\in\X\) is a given point, then \(P(E,\partial B_r(x))=0\) for all but countably many radii \(r>0\).
This is due to the fact that \(\partial B_r(x)\cap\partial B_s(x)=\varnothing\) whenever \(0<r<s\) {and to the \(\sigma\)-finiteness of \(P(E,\cdot)\).}
\fr}\end{remark}

Given any \(f\in\LIP_{\mathrm{loc}}(\X)\), it holds that \(|{\bf D}f|\) is a locally finite measure and \(|{\bf D}f|\leq\lip(f)\mm\).
\begin{theorem}[Coarea formula {\cite[Proposition 4.2]{MIRANDA2003}}]
Let \((\X,\sfd,\mm)\) be a metric measure space. Fix any \(f\in L^1_{\mathrm{loc}}(\X)\) such that
\(|{\bf D}f|\) is a locally finite measure. Fix a Borel set \(E\subseteq\X\). Then
\(\R\ni t\mapsto P(\{f<t\},E)\in[0,+\infty]\) is a Borel measurable function and it holds that
\[
|{\bf D}f|(E)=\int_\R P(\{f<t\},E)\,\d t.
\]
\end{theorem}
\begin{cor}\label{cor:conseq_coarea}
Let \((\X,\sfd,\mm)\) be a metric measure space. Fix \(x\in\X\) and a Borel set \(E\subseteq\X\).
Define \(f\colon(0,+\infty)\to\R\) as \(f(r)\coloneqq|{\bf D}\sfd_x|(E\cap B_r(x))\) for every \(r>0\), where we denote
\(\sfd_x\coloneqq\sfd(\cdot,x)\in\LIP(\X)\). Then the function \(f\) is locally absolutely continuous and it holds
that \(f'(r)=P(B_r(x),E)\) for \(\mathcal L^1\)-a.e.\ \(r>0\).
\end{cor}
\begin{proof}
By virtue of the coarea formula, we obtain that \(f(r)=\int_\R P\big(\{\sfd_x<s\},E\cap B_r(x)\big)\,\d s=\int_0^r P(B_s(x),E)\,\d s\)
for every \(r>0\), whence it follows that \(f(r)-f(\tilde r)=\int_{\tilde r}^r P(B_s(x),E)\,\d s\) for every \(r>\tilde r>0\).
Hence, \(f\) is locally absolutely continuous and \(f'(r)=P(B_r(x),E)\) for every Lebesgue point \(r\) of
\(s\mapsto P(B_s(x),E)\), thus for \(\mathcal L^1\)-a.e.\ \(r>0\).
\end{proof}
\subsection{PI spaces}\label{s:PI}
Even though the general theory of sets of finite perimeter is meaningful in any metric measure space, a much more refined calculus is
available in the class of doubling spaces supporting a weak form of \((1,1)\)-Poincar\'{e} inequality, which we refer to as \emph{PI spaces}.
Below we recall the definition of PI space we adopt in this paper, referring e.g.\ to \cite{Bjorn-Bjorn11,HKST15} for a thorough account of this topic.
We will also recall some key features of sets of finite perimeter in PI spaces.
\begin{definition}[PI space]\label{def:PI}
Let \((\X,\sfd,\mm)\) be a metric measure space. Then:
\begin{itemize}
\item We say that \((\X,\sfd,\mm)\) is \emph{uniformly locally doubling} if there is a function
\(C_D\colon(0,+\infty)\to(0,+\infty)\) such that
\[
\mm\big(B_{2r}(x)\big)\leq C_D(R)\,\mm\big(B_r(x)\big)\quad\text{ for every }0<r<R\text{ and }x\in\X.
\]
\item We say that \((\X,\sfd,\mm)\) supports a \emph{weak local \((1,1)\)-Poincar\'{e} inequality} if
there exist a constant \(\lambda\geq 1\) and a function \(C_P\colon(0,+\infty)\to(0,+\infty)\) such that
for any function \(f\in\LIP_{\mathrm{loc}}(\X)\) it holds that
\[
\fint_{B_r(x)}\bigg|f-\fint_{B_r(x)}f\,\d\mm\bigg|\,\d\mm\leq C_P(R)\,r\fint_{B_{\lambda r}(x)}\lip(f)\,\d\mm
\quad\text{ for all }0<r<R\text{ and }x\in\X.
\]
\item \((\X,\sfd,\mm)\) is a \emph{PI space} if it is uniformly locally doubling
and it supports a weak local \((1,1)\)-Poincar\'{e} inequality.
\end{itemize}
\end{definition}
{
We point out that if \((\X,\sfd,\mm)\) is a uniformly locally doubling space, then \((\X,\sfd)\) is proper,
so \((\X,\sfd)\) is locally compact, and locally finite Borel measures on \((\X,\sfd)\) are boundedly finite.
}
\begin{remark}\label{rem:StrongPoincare}{\rm
Let \((\X,\sfd,\mm)\) be a PI space such that \((\X,\sfd)\) is a \emph{length space}, i.e.\ the distance between
any two points in \(\X\) is the infimum of the lengths of rectifiable curves joining them. Then
the weak local \((1,1)\)-Poincar\'{e} inequality is in fact \emph{strong}, namely it holds with \(\lambda=1\);
see for example \cite[Corollary 9.5 and Theorem 9.7]{hajkosk}.
{Moreover, the completeness and the local compactness of \((\X,\sfd)\) ensure that \((\X,\sfd)\) is also geodesic.}
\fr}\end{remark}

Given a Borel set \(E\subseteq\X\) in a PI space \((\X,\sfd,\mm)\), we define its \emph{essential interior} and \emph{essential exterior} as
\[
E^{(1)}\coloneqq\bigg\{x\in\X\;\bigg|\;\lim_{r\to 0}\frac{\mm(E\cap B_r(x))}{\mm(B_r(x))}=1\bigg\},\qquad
E^{(0)}\coloneqq\bigg\{x\in\X\;\bigg|\;\lim_{r\to 0}\frac{\mm(E\cap B_r(x))}{\mm(B_r(x))}=0\bigg\},
\]
respectively. The \emph{essential boundary} of \(E\) is defined as \(\partial^e E\coloneqq\X\setminus(E^{(1)}\cup E^{(0)})\).
Notice that \(E^{(1)}\), \(E^{(0)}\), and \(\partial^e E\) are Borel sets and that \(\partial^e E\subseteq\partial E\).
It follows from the Lebesgue differentiation theorem (which holds on every uniformly locally doubling metric measure space,
see e.g.\ \cite[Theorem 1.8]{Heinonen01}) that \(\mm(E^{(1)}\Delta E)=0\) and \(\mm(E^{(0)}\Delta(\X\setminus E))=0\). Moreover,
if \(E\) is a set of finite perimeter, then we know from \cite[Theorem 5.3]{amb01} that \(P(E,\cdot)\) is concentrated
on \(\partial^e E\).
\begin{proposition}\label{prop:formula_union}
Let \((\X,\sfd,\mm)\) be a PI space. Let \(E,F\subseteq\X\) be sets of locally finite perimeter with \(P(E,\partial^e F)=0\). Then
\[
P(E\cap F,\cdot)\leq P(E,\cdot) \res F^{(1)} +P(F,\cdot) \res E^{(1)}.
\]
\end{proposition}
\begin{proof}
We know from \cite[Theorem 5.3]{amb01} that the perimeter measure \(P(G,\cdot)\) of a set \(G\subseteq\X\) of locally finite perimeter
can be written as \(P(G,\cdot)=\theta_G\mathcal H^h|_{\partial^e G}\) for some Borel function \(\theta_G\colon\X\to(0,+\infty)\),
where \(\mathcal H^h\) stands for the codimension-one Hausdorff measure (see \cite[Section 5]{amb01}). Since
\(P(E\cap F,\cdot)\leq P(E,\cdot)+P(F,\cdot)\) and \(P(E,\cdot)|_{\X\setminus\partial^e E}=P(F,\cdot)|_{\X\setminus\partial^e F}=0\),
we deduce that \(\theta_{E\cap F}\leq\theta_E\) and \(\theta_{E\cap F}\leq\theta_F\) hold \(\mathcal H^h\)-a.e.\ in
\(\partial^e E\setminus\partial^e F\) and \(\partial^e F\setminus\partial^e E\), respectively. Moreover, we deduce from
\(\int_{\partial^e F}\theta_E\,\d\mathcal H^h|_{\partial^e E}=P(E,\partial^e F)=0\) that \(\mathcal H^h(\partial^e E\cap\partial^e F)=0\).
Given that \(\partial^e(E\cap F)=(\partial^e E\cap F^{(1)})\sqcup(\partial^e F\cap E^{(1)})\) up to an $\mathcal H^h$-negligible set, which is shown e.g.\ in the
proof of \cite[Lemma 2.5]{AntonelliPasqualettoPozzetta21}, we conclude that
\[\begin{split}
P(E\cap F, \cdot)&=\theta_{E\cap F}\mathcal H^h|_{\partial^e(E\cap F)}=
\theta_{E\cap F}\mathcal H^h|_{\partial^e E\cap F^{(1)}}+\theta_{E\cap F}\mathcal H^h|_{\partial^e F\cap E^{(1)}}
\\&\leq\theta_E\mathcal H^h|_{\partial^e E\cap F^{(1)}}+\theta_F\mathcal H^h|_{\partial^e F\cap E^{(1)}},
\end{split}\]
which yields the statement.
\end{proof}

The following is a direct consequence of the study in \cite{amb01}, taking Remark \ref{rem:StrongPoincare} into account.
\begin{theorem}[Relative isoperimetric inequality {\cite[Remark 4.4]{amb01}}]\label{thm:isoper_ineq}
Let \((\X,\sfd,\mm)\) be a length PI space. Then there exists a function
\(C_I=C_I(C_D,C_P)\colon(1,+\infty)\times(0,+\infty)\to(0,+\infty)\) such that the following property holds: given a set
\(E\subseteq\X\) of finite perimeter, a radius \(R>0\), and an exponent \(\alpha>\max\{\log_2(C_D(R)),1\}\), we have that
\[
\min\big\{\mm(B_r(x)\cap E),\mm(B_r(x)\setminus E)\big\}\leq
C_I(\alpha,R)\bigg(\frac{r^\alpha}{\mm\big(B_r(x)\big)}\bigg)^{\frac{1}{\alpha-1}}P\big(E,B_r(x)\big)^{\frac{\alpha}{\alpha-1}},
\]
for every \(x\in\X\) and \(r\in(0,R)\).
\end{theorem}
In the next proposition we recall the well-known fact that in the class of PI spaces where unit balls have measure uniformly
bounded away from zero, there holds an isoperimetric inequality for sets of small volume. Such a result is essentially due to
\cite{CoulhonSaloffCoste95}, after \cite{Buser82, Kanai86, ChavelFeldman91}. For a proof, we refer the reader to the argument
in \cite[Lemma V.2.1]{ChavelIsoperimetricBook01}.
\begin{proposition}[Isoperimetric inequality for small volumes]\label{thm:isop_small}
Let \((\X,\sfd,\mm)\) be a length PI space. Then there exist constants $\alpha>1$, $C>0$ such that the following holds.
If $v_0\coloneqq \inf_{x \in \X} \mea(B_1(x))>0$, then for all Borel sets $E\subseteq\X$ with $\mea(E)<v_0/2$ it holds that
\[
P(E)\ge C v_0^\frac{1}{\alpha} \mea(E)^{\frac{\alpha-1}{\alpha}}.
\]
\end{proposition}
\begin{remark}\label{rmk:P_non_null}{\rm
Let \((\X,\sfd,\mm)\) be a length PI space and \(E\subseteq\X\) a set of finite perimeter such that \(\mm(E),\mm(\X\setminus E)>0\).
Then the relative isoperimetric inequality ensures that \(P(E)\neq 0\). In order to prove it, fix any \(x\in\X\) and notice that
we have \(\mm(B_R(x)\cap E),\mm(B_R(x)\setminus E)>0\) for some \(R>0\) sufficiently large, thus \(P(E)\geq P(E,B_R(x))>0\).
\fr}\end{remark}
\begin{lemma}\label{lem:var_dist}
Let \((\X,\sfd,\mm)\) be a length PI space. Then there exists \(c=c(\inf C_D,\inf C_P)\in(0,1)\) such that
\begin{equation}\label{eq:total variation and m}
c\,\mm\leq|{\bf D}\sfd_x|\leq\mm\quad\text{ for every }x\in\X,
\end{equation}
where we denote \(\sfd_x\coloneqq\sfd(\cdot,x)\in\LIP_{\mathrm{loc}}(\X)\). In particular, it holds that
\begin{equation}\label{eq:null ball boundary}
\mm\big(\partial B_r(x)\big)=0\quad\text{ for every }x\in\X\text{ and }r>0.
\end{equation}
\end{lemma}
\begin{proof}
Recall that \(|{\bf D}\sfd_x|\leq\lip(\sfd_x)\mm\). Moreover, we deduce from \cite[Equation (4.5)]{Ambrosio-Pinamonti-Speight15}
that there exists a constant \(c=c(\inf C_D,\inf C_P)\in(0,1)\) such that \(c\,\lip(\sfd_x)\mm\leq|{\bf D}\sfd_x|\). To obtain
\eqref{eq:total variation and m}, observe that \(\lip(\sfd_x)\equiv 1\): the inequality \(\lip(\sfd_x)\leq 1\) holds
in any metric space, while the converse inequality readily follows from the {fact that \((\X,\sfd)\) is geodesic}.
Finally, \eqref{eq:null ball boundary} can be proved by combining \eqref{eq:total variation and m} with the coarea formula:
we can estimate
\[
\mm\big(\partial B_r(x)\big)\leq\frac{1}{c}|{\bf D}\sfd_x|\big(\partial B_r(x)\big)=
\frac{1}{c}\int_\R P\big(B_s(x),\partial B_r(x)\big)\,\d s=0,
\]
where the last identity follows from the fact that \(P(B_s(x),\cdot)\) is concentrated on \(\partial^e B_s(x)\subseteq\partial B_s(x)\).
\end{proof}
\section{Topological regularity}
Let us begin with the definition of a volume-constrained minimizer of the perimeter.
\begin{definition}[Volume-constrained minimizer]\label{def:VolumeConstrained}
Let \((\X,\sfd,\mm)\) be a metric measure space. Then a set \(E\subseteq\X\) of locally finite perimeter is said to be
a \emph{volume-constrained minimizer} of the perimeter if the following property is verified: given a Borel
set \(F\subseteq\X\) and a compact set \(K\subseteq\X\) satisfying \(\mm((E\Delta F)\setminus K)=0\)
and \(\mm(E\cap K)=\mm(F\cap K)\), it holds \(P(E, K)\leq P(F, K)\).
\end{definition}

Observe that \(E\) is a volume-constrained minimizer if and only if \(\X\setminus E\) is a volume-constrained minimizer.
\begin{remark}{\rm
An \emph{isoperimetric set}, i.e.\ a set \(E\subseteq\X\) of finite perimeter with \(0<\mm(E)<+\infty\)
such that \(P(E)\leq P(F)\) for any Borel set \(F\subseteq\X\) with \(\mm(F)=\mm(E)\), is a volume-constrained
minimizer of the perimeter.
\fr}\end{remark}
Next we introduce our definition of a metric measure space having the \emph{deformation property}, which will be our standing
assumption throughout the rest of the paper.
\begin{definition}[Deformation property]\label{def:deform}
Let \((\X,\sfd,\mm)\) be a metric measure space {with \((\X,\sfd)\) proper.} Then we say that \((\X,\sfd,\mm)\)
has the \emph{deformation property} if the following property holds: for every set of locally finite perimeter \(E\subseteq\X\)
and any point \(x\in\X\), there exist constants \(R\in(0,1]\) and \(C\ge 0\) such that
\begin{subequations}\begin{align}
\label{eq:deform_1}
P(E\setminus B_r(y),{B_{2R}(x)})\leq C\frac{\mm(B_r(y)\cap E)}{r}+P(E,{B_{2R}(x)})&\quad\forall y\in B_R(x),\,r\in(0,R),\\
\label{eq:deform_2}
P(E\cup B_r(y),{B_{2R}(x)})\leq C\frac{\mm(B_r(y)\setminus E)}{r}+P(E,{B_{2R}(x)})&\quad\forall y\in B_R(x),\,r\in(0,R),
\end{align}\end{subequations}
\end{definition}
%for any bounded closed set $K\subseteq\X$ such that $B_{2R}(x)\subseteq K$.
%
%

For convenience, we define from now on $R_x(E)\in(0,1]$ to be the maximal \(R\in(0,1]\)  such that the above holds for some
$C\ge 0$ and we  define $C_x(E)\ge 0$ to be the minimal constant such that \eqref{eq:deform_1} and \eqref{eq:deform_2} hold
with $R=R_x(E).$ Note that, by symmetry, we have that \(R_x(E)=R_x(\X\setminus E)\) and \(C_x(E)=C_x(\X\setminus E)\);
this is the reason why in Definition \ref{def:deform} we require the validity of both \eqref{eq:deform_1} and
\eqref{eq:deform_2} with the same constants \(C\) and \(R\).
{
We also observe that if \(E\subseteq\X\) is a given set of finite perimeter (resp.\ of locally finite perimeter), then \eqref{eq:deform_1} is equivalent to asking that
\(P(E\setminus B_r(y),S)\leq C\frac{\mm(B_r(y)\cap E)}{r}+P(E,S)\) holds for every \((y,r)\in B_R(x)\times(0,R)\) and every Borel set (resp.\ bounded Borel set)
\(S\subseteq\X\) with \(B_{2R}(x)\subseteq S\). Similarly for \eqref{eq:deform_2}. We will often make use of this observation without further notice. Also:
\begin{equation}\label{eq:R_x_unif_cpt}
\inf_{x\in B}R_x(E)>0\quad\text{ for every bounded set }B\subseteq\X.
\end{equation}
Indeed, the compactness of the closure of \(B\) ensures that \(B\subseteq\bigcup_{i=1}^n B_{R_{x_i}(E)/2}(x_i)\) for some \(x_1,\ldots,x_n\in B\),
which gives \(R_x(E)\geq\delta\coloneqq\min\big\{R_{x_i}(E)/2\,:\,i=1,\ldots,n\big\}>0\) for every \(x\in B\). The same argument shows also that
\eqref{eq:deform_1} and \eqref{eq:deform_2} hold for every \(x\in B\) for some \(R\) and \(C\) that depend only on \(B\) and \(E\),
e.g.\ by taking \(R\coloneqq\delta\) and \(C\coloneqq\max\big\{C_{x_i}(E)\,:\,i=1,\ldots,n\big\}\).
}
\begin{remark}[Spaces having the deformation property]\label{rmk:ex_deform_prop}{\rm
These are some spaces with the deformation property:
\begin{itemize}
\item[\(\rm i)\)] Euclidean spaces (see e.g.\ \cite{GonzalezMassariTamanini} and the references therein).
\item[\(\rm ii)\)] Riemannian manifolds (this can be proved e.g.\ by following the proof of \cite[Theorem 1.1]{AntonelliPasqualettoPozzetta21}
and using the fact that the Ricci curvature is locally bounded from below).
\item[\(\rm iii)\)] \({\sf RCD}(K,N)\) spaces with \(K\in\R\) and \(N\in[1,\infty)\) ({proved in \cite[Theorem 1.1]{AntonelliPasqualettoPozzetta21}
building upon the Gauss--Green formula in \cite[Theorem 2.4]{BPS23}}).
\end{itemize}
We point out that in the above cases a stronger version of the deformation property holds, since, given an arbitrary $R>0$,
the constants \(C_x(E)\) for which the deformation property holds at every point $x\in\X$ and for every $0<r<R$, can be chosen
to be independent of \(E,x\), and to be dependent only on $K,N,R$.

It would be interesting to study whether there are other distinguished examples of PI spaces having the deformation property.
One natural class to investigate is the one of sub-Riemannian manifolds, or, more specifically, the one of Carnot groups.
For example, in the first Heisenberg group one has a sub-Laplacian comparison theorem. Being $r$ the Carnot--Carathéodory distance
from the origin, we have that $\Delta_{\mathrm{H}}r\leq 4/r$ holds in the distributional sense, where $\Delta_{\mathrm{H}}$ is the
horizontal Laplacian. See \cite{SubLaplacianComparison} for the study of sub-Laplacian comparison theorems in more general
sub-Riemannian structures, and \cite[Corollary 4.19]{CavallettiMondinoAnlPDE} for the Laplacian comparison theorem
in arbitrary essentially non-branching $\sf MCP$ spaces. Then, coupling this with the Gauss--Green formulae for Carnot groups
in \cite{ComiMagnani}, one could argue following the lines of \cite[Theorem 2.32 and Theorem 1.1]{AntonelliPasqualettoPozzetta21}
to obtain that at least $\mathbb H^1$, and more in general all the groups that are essentially non-branching
${\sf MCP}(K,N)$ spaces, with $K\in\mathbb R$ and $N\in (1,\infty)$ (cf. \cite{BarilariRizziMCP}, \cite{BadreddineRifford}), have the deformation property. Since this is out of the scope
of the present note, and since there are also some regularity issues of the distance function to deal with, we do not treat these
examples here, but we leave it to possible future investigations.

We mention that, on the other hand, the topological regularity of isoperimetric sets is already proved in \cite{LeonardiRigot}
in the setting of Carnot groups and in \cite{GalliRitore} on a certain class of sub-Riemannian manifolds.
\fr}\end{remark}
\begin{remark}\label{rmk:example_PI_no_def}{\rm
There exist PI spaces where the deformation property fails. For example, fix a sequence of pairwise well-separated non-empty balls \(B_n\coloneqq B_{r_n}(x_n)\) in \(\R^2\) such that \(x_n\to 0\)
and \(\sum_n r_n<+\infty\). Now consider the density function \(\rho\colon\R^2\to[1,2]\) given by \(\rho\coloneqq\nchi_E+2\nchi_{\R^2\setminus E}\), where \(E\coloneqq\bigcup_n B_n\).
Letting \(\mm\coloneqq\rho\mathcal L^2\) we have \(\mathcal L^2\leq\rho\mathcal L^2\leq 2\mathcal L^2\), so that \((\R^2,|\cdot|,\mm)\) is an Ahlfors regular geodesic PI space.
We claim that the deformation property is not valid for the set of finite perimeter \(E\) at the origin \(0\). To check it, notice that for any \(n\in\N\) it holds \(P(B_n)=2\pi r_n\), while
\(P(B_{r_n+\varepsilon}(x_n))=4\pi(r_n+\varepsilon)\) and \(\mm(B_{r_n+\varepsilon}(x_n)\setminus B_n)=2\pi(2r_n\varepsilon+\varepsilon^2)\) for any $\varepsilon \in (0,\varepsilon_n)$ for some $\varepsilon_n>0$ sufficiently small. Therefore,
\[
\frac{P(E\cup B_{r_n+\varepsilon}(x_n))-P(E)}{\mm(B_{r_n+\varepsilon}(x_n)\setminus E)/(r_n+\varepsilon)}=
\frac{(2\pi r_n+4\pi\varepsilon)(r_n+\varepsilon)}{2\pi(2r_n\varepsilon+\varepsilon^2)}\to+\infty\quad\text{ as }\varepsilon\searrow 0,
\]
which shows that the deformation property fails at the origin.
%{
%If \(2\sum_{n\in S}r_n^2<\big(\sum_{n\in S}r_n\big)^2\) for some \(S\subseteq\N\), then \(E\) is not an isoperimetric set: compare \(E\) with \((E\cup B)\setminus\bigcup_{n\in S}B_{r_n}(x_n)\),
%where \(B\) is any ball with \(\mm(B)=\mm\big(\bigcup_{n\in S}B_{r_n}(x_n)\big)\) that is well-separated from \(E\).}
However, we are not aware of any example of a PI space where the deformation property fails when tested on an isoperimetric set, nor of an example of a PI space where the essential interior of some isoperimetric set is not topologically open.
%We suspect that \(E\) is an isoperimetric set if \(2\sum_{n\in S}r_n^2\geq\big(\sum_{n\in S}r_n\big)^2\) for every \(S\subseteq\N\), which happens e.g.\ if \(\sum_{k>n}r_k<(\sqrt 2-1)r_n\) for all \(n\in\N\).
%
%However, comparing with a ball $B$ such that $\mm(B)=\mm(E)$ and $B\cap E=\varnothing$, one checks that \(E\) cannot be an isoperimetric set.
%
\fr}\end{remark}
\begin{remark}\label{rem:Exaple2PINoDef}{\rm
The validity of the deformation property on a metric measure space $(\X, \sfd,\mm)$ entails a growth condition: given \(x\in\X\), there exist \(C_x,r_x>0\) such that
\begin{equation}\label{eq:growth_cond}
P(B_r(y))\leq C_x\frac{\mm(B_r(y))}{r}\quad\text{ for every }y\in B_{r_x}(x)\text{ and }r\in(0,r_x).
\end{equation}
The previous \eqref{eq:growth_cond} follows just by taking \(E\coloneqq\varnothing\) in the deformation property. We have that \eqref{eq:growth_cond} is not equivalent to the deformation property (e.g.\ in the example in
Remark \ref{rmk:example_PI_no_def} the property \eqref{eq:growth_cond} is satisfied). However, there are examples of PI spaces where also \eqref{eq:growth_cond} fails.
The example we are going to describe has been pointed out to the authors by Panu Lahti.
Consider the measure \(\mm\coloneqq|x|^{-1/2}\d x\) in \(\R\). Since the function \(|x|^{-1/2}\)
is an \(A_1\)-Muckenhoupt weight, we know that \((\R,|\cdot|,\mm)\) is a PI space (see e.g.\ \cite{HKM06}). Using that \(\frac{\mm(B_r(0))}{r}=2\fint_0^r\frac{1}{\sqrt x}\,\d x=\frac{4}{\sqrt r}\to+\infty\)
as \(r\searrow 0\), one can easily check that the codimension-one Hausdorff measure of the singleton $\{0\}$ diverges, i.e.\ \(\mathcal H^h(\{0\})=+\infty\). It follows from \cite[Theorem 5.3]{amb01} that \(B_{|y|}(y)\) is not a set of locally finite perimeter when \(y\in(0,+\infty)\).
Hence, \eqref{eq:growth_cond} fails for \(x=0\).
\fr}\end{remark}
Given a metric measure space $(\X,\sfd,\mm)$, a point $x\in \X$, and a Borel set $E\subseteq \X$, we introduce the notation
\begin{equation}\label{eq:aux_v_w}
 v_{E,x}(r)\coloneqq \mea(B_r(x)\cap E),\quad w_{E,x}(r)\coloneqq \mea(B_r(x)\setminus E) \quad \text{ for every }r>0.
\end{equation}

The core of the proof of our main Theorem \ref{thm:main} is contained in the following technical result.
\begin{lemma}\label{lem:key_lemma}
Let \((\X,\sfd,\mm)\) be a length PI space having the deformation property. Let \(E\subseteq\X\) be a volume-constrained
minimizer of the perimeter. Fix any \(x\in E^{(0)}\) and \(y\in E^{(1)}\). Define the functions
\(v_{E,x},w_{E,y}\colon(0,+\infty)\to[0,+\infty)\) as in \eqref{eq:aux_v_w}. Fix a sequence \((r_n)_n\subseteq(0,1)\)
such that \(r_n\to 0\). For any \(n\in\N\), we define the Borel set \(A_{E,r_n}^{x,y}\subseteq(0,r_n)\) as
\begin{equation}\label{eq:aux_A}
A_{E,r_n}^{x,y}\coloneqq\big\{r\in(0,r_n)\;\big|\;v_{E,x}(r)\geq w_{E,y}(r)\big\}.
\end{equation}
Suppose the following conditions are verified:
\begin{itemize}
\item[\(\rm i)\)] There exists \(\delta\in (0,R_y(E))\) such that \(\bar B_\delta(x)\cap\bar B_\delta(y)=\varnothing\),
\(v_{E,x}(\delta)>0\), and \(w_{E,y}(\delta)>0\).
\item[\(\rm ii)\)] The inequality \(\mathcal L^1(A_{E,r_n}^{x,y})\geq r_n/2\) holds for infinitely many \(n\in\N\).
\end{itemize}
Then it holds that \(x\in{\rm int}(E^{(0)})\).
\end{lemma}
\begin{proof}
We argue by contradiction: suppose that \(x\notin{\rm int}(E^{(0)})\). Recalling that
\(\mm(B_\delta(y)\setminus E)=w_{E,y}(\delta)>0\) and noticing that \(\mm(B_r(x)\cap E)\to 0\) as \(r\to 0\),
we can extract a (not relabeled) subsequence of \((r_n)_n\) for which
\begin{equation}\label{eq:key_lemma_aux}
r_n<\delta,\qquad\mm(B_{r_n}(x)\cap E)<\mm(B_\delta(y)\setminus E),\qquad\mathcal L^1(A_{E,r_n}^{x,y})\geq\frac{r_n}{2},
\end{equation}
for every \(n\in\N\). Now let \(n\in\N\) be fixed. We claim that for any \(r\in A_n\coloneqq A_{E,r_n}^{x,y}\) there exists
\(s(r)\in[r,\delta)\) such that
\[
v_{E,x}(r)=\mm(B_r(x)\cap E)=\mm(B_{s(r)}(y)\setminus E)=w_{E,y}(s(r)).
\]
Indeed, if \(w_{E,y}(r)=v_{E,x}(r)\), then we can take \(s(r)\coloneqq r\). If \(w_{E,y}(r)\neq v_{E,x}(r)\),
then \(v_{E,x}(r)>w_{E,y}(r)\) by definition of \(A_n\), thus the continuity of \(w_{E,y}\) (which follows from
\eqref{eq:null ball boundary}) ensures that \(w_{E,y}(s(r))=v_{E,x}(r)\) for some \(s(r)>0\). Since \(w_{E,y}\) is non-decreasing,
we infer that \(s(r)\geq r\). Moreover, the second inequality in \eqref{eq:key_lemma_aux} implies that \(s(r)<\delta\).

Given any \(r\in A_n\), we define the Borel set \(E_r\subseteq\X\) as \(E_r\coloneqq(E\setminus B_r(x))\cup B_{s(r)}(y)\).
The first inequality in \eqref{eq:key_lemma_aux} ensures that \(\bar B_r(x)\cap\bar B_{s(r)}(y)=\varnothing\), whence it
follows that \(\mm\big(E_r\cap(\bar B_r(x)\cup\bar B_{s(r)}(y))\big)=\mm\big(E\cap(\bar B_r(x)\cup\bar B_{s(r)}(y))\big)\).
{
Denote $K\coloneqq\bar B_{2\delta}(x) \cup \bar B_{2\delta}(y)$ for brevity.
}
%Let $K\subseteq \X$ be a compact set such that $\bar B_{2\delta}(x) \cup \bar B_{2\delta}(y) \subseteq K$.
The assumption that \(E\) is a volume-constrained minimizer of the perimeter then implies that \(P(E,K)\leq P(E_r,K)\). For ease
of notation from now on we will denote $C_y(E)$ simply by $C_y$. Thanks to Proposition \ref{prop:formula_union}, Remark
\ref{rmk:bdry_null}, the deformation property, and \(s(r)\geq r\), we deduce that for \(\mathcal L^1\)-a.e.\ \(r\in A_n\) one has
\[\begin{split}
P(E,K)&\leq P(E_r,K) \\ &\leq P\big(E\cup B_{s(r)}(y),B_r(x)^{(0)}\cap K\big)+P\big(B_r(x),(E\cup B_{s(r)}(y))^{(1)}\cap K\big)\\
&= P(E\cup B_{s(r)}(y),K)-P\big(E\cup B_{s(r)}(y),\partial^e B_r(x)\cup B_r(x)^{(1)}\big)+P(B_r(x),E^{(1)})\\
&\leq P(E\cup B_{s(r)}(y),K)-P(E,B_r(x))+P(B_r(x),E^{(1)})\\
&\leq P(E,K)+C_y\frac{\mm(B_{s(r)}(y)\setminus E)}{s(r)}-P(E,B_r(x))+P(B_r(x),E^{(1)})\\
&=P(E,K)+C_y\frac{\mm(B_r(x)\cap E)}{s(r)}-P(E,B_r(x))+P(B_r(x),E^{(1)})\\
&\leq P(E,K)+C_y\frac{\mm(B_r(x)\cap E)}{r}-P(E,B_r(x))+P(B_r(x),E^{(1)}).
\end{split}\]
Notice that the constant \(C_y\) depends on \(y\) and \(E\), but neither on \(n\) nor on \(r\). Therefore, we have shown that
\begin{equation}\label{eq:key_lemma_aux2}
P(E,B_r(x))\leq C_y\frac{\mm(B_r(x)\cap E)}{r}+P(B_r(x),E^{(1)})\quad\text{ for all }n\in\N\text{ and }\mathcal L^1\text{-a.e.\ }r\in A_n.
\end{equation}
Now fix any \(\alpha>\max\{\log_2(C_D(\delta)),1\}\). We know from the relative isoperimetric inequality, i.e.\ Theorem \ref{thm:isoper_ineq}, that
\begin{equation}\label{eq:key_lemma_aux3}
P(E,B_r(x))\geq 2\tilde C\min\{v_{E,x}(r),w_{E,x}(r)\big\}^{1-\frac{1}{\alpha}}\frac{\mm(B_r(x))^\frac{1}{\alpha}}{r}
\quad\text{ for every }n\in\N\text{ and }r\in A_n,
\end{equation}
where we define \(\tilde C\coloneqq 1/2 C_I(\alpha,\delta)^{(\alpha-1)/\alpha}\). Exploiting the fact that \(x\in E^{(0)}\),
we can find \(\bar n\in\N\) such that
\begin{equation}\label{eq:key_lemma_aux4}
v_{E,x}(r)<w_{E,x}(r),\qquad C_y\bigg(\frac{\mm(B_r(x)\cap E)}{\mm(B_r(x))}\bigg)^{\frac{1}{\alpha}}\leq\tilde C
\quad\text{ for every }n\geq\bar n\text{ and }r\in A_n.
\end{equation}
By combining \eqref{eq:key_lemma_aux2}, \eqref{eq:key_lemma_aux3}, and \eqref{eq:key_lemma_aux4}, we deduce that
for every \(n\geq\bar n\) and \(\mathcal L^1\)-a.e.\ \(r\in A_n\) it holds that
\[\begin{split}
2\tilde C\mm(B_r(x)&\cap E)^{1-\frac{1}{\alpha}}\frac{\mm(B_r(x))^{\frac{1}{\alpha}}}{r}\\&\leq P(B_r(x),E^{(1)})+
\mm(B_r(x)\cap E)^{1-\frac{1}{\alpha}}C_y\bigg(\frac{\mm(B_r(x)\cap E)}{\mm(B_r(x))}\bigg)^{\frac{1}{\alpha}}\frac{\mm(B_r(x))^{\frac{1}{\alpha}}}{r}\\
&\leq P(B_r(x),E^{(1)})+\tilde C\mm(B_r(x)\cap E)^{1-\frac{1}{\alpha}}\frac{\mm(B_r(x))^{\frac{1}{\alpha}}}{r}.
\end{split}\]
Rearranging the terms, we infer that
\begin{equation}\label{eq:key_lemma_aux5}
\tilde C\frac{\mm(B_r(x))^{\frac{1}{\alpha}}}{r}\mm(B_r(x)\cap E)^{1-\frac{1}{\alpha}}\leq P(B_r(x),E^{(1)})
\quad\text{ for every }n\geq\bar n\text{ and }\mathcal L^1\text{-a.e.\ }r\in A_n.
\end{equation}
Now define the function \(f\colon(0,+\infty)\to\R\) as \(f(r)\coloneqq|{\bf D}\sfd_x|(B_r(x)\cap E^{(1)})\) for every \(r>0\).
Corollary \ref{cor:conseq_coarea} tells that \(f\) is locally absolutely continuous and \(f'(r)=P(B_r(x),E^{(1)})\) for
\(\mathcal L^1\)-a.e.\ \(r>0\). Moreover, Lemma \ref{lem:var_dist} gives \(f(r)\leq\mm(B_r(x)\cap E)\) for every \(r>0\).
Consequently, it follows from \eqref{eq:key_lemma_aux5} that
\begin{equation}\label{eq:key_lemma_aux6}
\tilde C\frac{\mm(B_r(x))^{\frac{1}{\alpha}}}{r}f(r)^{1-\frac{1}{\alpha}}\leq f'(r)
\quad\text{ for every }n\geq\bar n\text{ and }\mathcal L^1\text{-a.e.\ }r\in A_n.
\end{equation}
Using {that \(x\notin{\rm int}(E^{(0)})\), which is the contradiction assumption, and Lemma \ref{lem:var_dist}} we see that
\(f(r)\geq c\,\mm(B_r(x)\cap E)>0\) for every \(r>0\), thus we can divide both sides of \eqref{eq:key_lemma_aux6}
by \(\alpha f(r)^{1-\frac{1}{\alpha}}\), obtaining that
\begin{equation}\label{eq:key_lemma_aux7}
\frac{\tilde C}{\alpha}\frac{\mm(B_r(x))^{\frac{1}{\alpha}}}{r}\leq\frac{f'(r)}{\alpha f(r)^{1-\frac{1}{\alpha}}}
=(f^{\frac{1}{\alpha}})'(r)\quad\text{ for every }n\geq\bar n\text{ and }\mathcal L^1\text{-a.e.\ }r\in A_n.
\end{equation}
The third inequality in \eqref{eq:key_lemma_aux} implies that \(\mathcal L^1([r_n/4,r_n]\cap A_n)\geq r_n/4\) for every \(n\in\N\),
thus integrating \eqref{eq:key_lemma_aux7} (and taking into account that \((f^{1/\alpha})'(r)\geq 0\) holds for
\(\mathcal L^1\)-a.e.\ \(r>0\)) we get that
\[\begin{split}
\frac{\tilde C}{4\alpha}\mm(B_{r_n/4}(x))^{\frac{1}{\alpha}}&\leq\frac{\tilde C}{\alpha}\frac{\mm(B_{r_n/4}(x))^{\frac{1}{\alpha}}}{r_n}\mathcal L^1([r_n/4,r_n]\cap A_n)
 \leq\frac{\tilde C}{\alpha}\int_{[r_n/4,r_n]\cap A_n}\frac{\mm(B_r(x))^{\frac{1}{\alpha}}}{r}\,\d r\\
&\leq\int_{[r_n/4,r_n]\cap A_n}(f^{\frac{1}{\alpha}})'(r)\,\d r \leq\int_0^{r_n}(f^{\frac{1}{\alpha}})'(r)\,\d r \\ & =f(r_n)^{\frac{1}{\alpha}} \leq\mm(B_{r_n}(x)\cap E)^{\frac{1}{\alpha}}
\end{split}\]
for every \(n\geq\bar n\). Letting \(C\coloneqq\frac{1}{C_D(\delta)^2}\big(\frac{\tilde C}{4\alpha}\big)^\alpha\), we can conclude
that \(\mm(B_{r_n}(x)\cap E)\geq C\mm(B_{r_n}(x))\) for every \(n\geq\bar n\). This leads to a contradiction with the fact that
\(x\in E^{(0)}\). Therefore, the proof of the statement is achieved.
\end{proof}

Having Lemma \ref{lem:key_lemma} at our disposal, we can now easily prove Theorem \ref{thm:main}.
\begin{proof}[Proof of Theorem \ref{thm:main}]
Since \(E^{(1)}=(\X\setminus E)^{(0)}\), it is sufficient to check that \(E^{(0)}={\rm int}(E^{(0)})\). To prove it,
we argue by contradiction: suppose there exists a point \(x\in E^{(0)}\setminus{\rm int}(E^{(0)})\). This implies that
both \(\mm(E)>0\) (otherwise \(E^{(0)}=\X={\rm int}(E^{(0)})\)) and \(\mm(\X\setminus E)>0\) (otherwise \(E^{(0)}=\varnothing\)),
thus we know from Remark \ref{rmk:P_non_null} that \(P(E)\neq 0\). Since \(P(E,\cdot)\) is concentrated on \(\partial^e E\),
we can find a point \(z\in\partial^e E\). Notice that \(\mm(B_r(z)\cap E)>0\) and \(\mm(B_r(z)\setminus E)>0\) for all \(r>0\).
Since \(z\neq x\), we can fix some radius \(\delta\in(0,R_x(E))\cap(0,2 R_z(E)/3)\cap(0,\sfd(x,z)/3)\). Thanks to the fact that
\(\mm(B_{\delta/2}(z)\cap E)>0\), we can find a point \(y\in E^{(1)}\cap B_{\delta/2}(z)\). Notice that
\(B_{\delta/2}(z)\setminus E\subseteq B_\delta(y)\setminus E\), so that
\(\mm(B_\delta(y)\setminus E)\geq\mm(B_{\delta/2}(z)\setminus E)>0\). The fact that \(x\notin{\rm int}(E^{(0)})\) implies
that also \(\mm(B_\delta(x)\cap E)>0\). Hence, letting \(v_{E,x}\), \(w_{E,y}\) be defined in \eqref{eq:aux_v_w}, we have
proved that \(v_{E,x}(\delta)>0\) and \(w_{E,y}(\delta)>0\). By our construction and by the definition of $R_z(E),R_y(E)$
it holds \(2 R_z(E)/3\le R_y(E)\), hence we have \(\delta\in(0,R_y(E))\). Moreover, the inequality \(\delta<\sfd(x,z)/3\)
implies that \(\sfd(x,y)>2\sfd(x,z)/3>2\delta\), which means that \(\bar B_\delta(x)\cap\bar B_\delta(y)=\varnothing\).
All in all, we showed that item i) of Lemma \ref{lem:key_lemma} holds. Hence, fixed any sequence \((r_n)_n\subseteq(0,1)\)
with \(r_n\to 0\), we deduce from the assumption \(x\notin{\rm int}(E^{(0)})\) that item ii) of Lemma \ref{lem:key_lemma} fails.
Letting \(A_{E,r_n}^{x,y}\) be as in \eqref{eq:aux_A}, we get that
\begin{equation}\label{eq:main_aux}
\mathcal L^1(A_{E,r_n}^{x,y})\geq\frac{r_n}{2}\quad\text{ holds only for finitely many }n\in\N.
\end{equation}
Since \(A_{E,r_n}^{x,y}\cup A_{\X\setminus E,r_n}^{y,x}=(0,r_n)\) for every \(n\in\N\), we infer that
\(\mathcal L^1(A_{\X\setminus E,r_n}^{y,x})\geq r_n/2\) for infinitely many \(n\in\N\). Given that
\(v_{\X\setminus E,y}(\delta)=w_{E,y}(\delta)>0\) and \(w_{\X\setminus E,x}(\delta)=v_{E,x}(\delta)>0\), we are in a position
to apply Lemma \ref{lem:key_lemma} again, obtaining that \(y\in{\rm int}((\X\setminus E)^{(0)})={\rm int}(E^{(1)})\).
This gives some \(\bar r>0\) satisfying \(w_{E,y}(r)=0\) for every \(r\in(0,\bar r)\). On the other hand, we know from
\(x\notin{\rm int}(E^{(0)})\) that \(v_{E,x}(r)>0\) for all \(r\in(0,\bar r)\). Choosing \(\bar n\in\N\) so that \(r_n<\bar r\)
for all \(n\geq\bar n\), we conclude that \(A_{E,r_n}^{x,y}=(0,r_n)\) for every \(n\geq\bar n\), in contradiction with
\eqref{eq:main_aux}. This proves that \(E^{(0)}={\rm int}(E^{(0)})\).
\end{proof}
\begin{remark}[Some generalizations of Theorem \ref{thm:main}]\label{rmk:generalizations}{\rm
To keep the presentation of Theorem \ref{thm:main} as clear as possible, we decided not to prove it in its utmost generality.
However, below we discuss some generalizations of our result that can be obtained by slightly adapting our arguments.
The standing assumption is that \((\X,\sfd,\mm)\) is a length PI space.
\begin{itemize}
\item[\(\rm i)\)] By inspecting the proof of Lemma \ref{lem:key_lemma}, one can see that assuming the validity of a weaker notion
of deformation property is sufficient. Namely, one can allow for the constant \(C\) appearing in \eqref{eq:deform_1},
\eqref{eq:deform_2} to depend on \(y\) and it is sufficient to require the deformation property only for volume-constrained
minimizers \(E\) of the perimeter.
\item[\(\rm ii)\)] A localized version of Theorem \ref{thm:main} holds as well: let \(E\subseteq\X\) be a volume-constrained
minimizer of the perimeter in some open set \(\Omega\subseteq\X\) (i.e.\ as in Definition \ref{def:VolumeConstrained} but requiring
that $K\subseteq\Omega$ and with $P(\cdot)$ replaced by $P(\cdot,\Omega)$) satisfying \(P(E,\Omega)>0\). Then
\(E^{(1)}\cap\Omega\), \(E^{(0)}\cap\Omega\) are open sets and \(\partial E^{(1)}\cap\Omega=\partial E^{(0)}\cap\Omega=\partial^e E\cap\Omega\).
\item[\(\rm iii)\)] Theorem \ref{thm:main} can be generalized to volume-constrained minimizers of a suitable class of
\emph{quasi-perimeters}. Fix an open set \(\Omega\subseteq\X\) and a functional \(G\colon\mathscr B(\Omega)\to\R\cup\{+\infty\}\)
with \(G(\varnothing)<+\infty\) having the following property: for any \(U\Subset\Omega\) open, there exist constants \(C=C(U)>0\)
and \(\sigma=\sigma(U)\in\big(1-\frac{1}{\max\{1,\log_2(\inf C_D)\}},1\big]\) such that
\[
G(E)\leq G(F)+C\mm(E\Delta F)^\sigma\quad\text{ whenever }E,F\in\mathscr B(\Omega)\text{ satisfy }E\Delta F\subseteq U.
\]
We then define the quasi-perimeter \(\mathscr P_G\) restricted to \(\Omega\) as
\(\mathscr P_G(E,\Omega)\coloneqq P(E,\Omega)+G(E\cap\Omega)\) for every \(E\in\mathscr B(\Omega)\).
Then an adaption of the previous arguments yields the validity of the following statement: if \(E\subseteq\X\)
is a volume-constrained minimizer of the quasi-perimeter \(\mathscr P_G\) in \(\Omega\) (i.e.\ as in Definition
\ref{def:VolumeConstrained} but requiring that $K\subseteq\Omega$, and with $P(\cdot)$ replaced by $\mathscr P_G(\cdot,\Omega)$)
satisfying \(P(E,\Omega)>0\), then \(E^{(1)}\cap\Omega\) and \(E^{(0)}\cap\Omega\) are open sets, and it holds that
\(\partial E^{(1)}\cap\Omega=\partial E^{(0)}\cap\Omega=\partial^e E\cap\Omega\).
\fr
\end{itemize}
}\end{remark}

Once we know that volume-constrained minimizers of the perimeter have an open representative, we can obtain the following
expected boundary density estimates by suitably adapting the arguments in the proof of Lemma \ref{lem:key_lemma}.
\begin{theorem}[Boundary density estimates]\label{thm:density}
Let \((\X,\sfd,\mm)\) be a length PI space having the deformation property. Let \(E\subseteq\X\) be a volume-constrained minimizer
of the perimeter. Let \(B\subseteq\X\) be a given bounded set. Then there exist constants \(\bar r=\bar r(E,B,C_D,C_I)>0\) and
\(C=C(E,B,C_D,C_I)>1\) such that
\begin{equation}\label{eq:density_claim}
\frac{1}{C}\leq\frac{\mm(B_r(x)\cap E)}{\mm(B_r(x))}\leq 1-\frac1C,\qquad\frac{1}{C}\leq\frac{r P(E,B_r(x))}{\mm(B_r(x))}
\leq C,
\end{equation}
for every $x\in\partial^e E\cap B$ and $r\in(0,\bar r)$.

In particular, there exists a constant \(\tilde C = \tilde C(C,C_D(\bar r/2)) \geq 1\) such that
\begin{equation}\label{eq:density_claim2}
P(E,B_{2r}(x))\leq\tilde C\,P(E,B_r(x))\quad\text{ for every }x\in\partial^e E\cap B\text{ and }r\in(0,\bar r/2).
\end{equation}
\end{theorem}
\begin{proof}
If \(\partial^e E\) contains only one point, the first one in \eqref{eq:density_claim} follows by the definition \(\partial^e E \),
while the second follows from \cite[Theorem 5.4]{amb01}. Thus we can assume that \(\partial^e E\) contains at least two distinct
points \(z\) and \(\tilde z\), otherwise there is nothing to prove. In particular, letting
\(\rho\coloneqq\min\big\{R_z(E),R_{\tilde z}(E),\frac{1}{5}\sfd(z,\tilde z)\big\}\),
we can find two points \(y\in B_{\rho/2}(z)\cap E^{(1)}\) and \(\tilde y\in B_{\rho/2}(\tilde z)\cap E^{(1)}\). In fact, Theorem
\ref{thm:main} ensures that \(y,\tilde y\in{\rm int}(E^{(1)})\), so that there exists \(r_0\in(0,\rho)\) such that
\begin{equation}\label{eq:density}
\mm(B_{r_0}(y)\setminus E)=\mm(B_{r_0}(\tilde y)\setminus E)=0.
\end{equation}
Notice that \(\mm(B_\rho(y)\setminus E)\geq\mm(B_{\rho/2}(z)\setminus E)>0\) and similarly \(\mm(B_\rho(\tilde y)\setminus E)>0\).
The doubling assumption ensures that the closure \(K\) of \(B\) is compact, thus an application of Dini's theorem yields the
existence of \(r_1>0\) such that
\begin{equation}\label{eq:density2}
\mm(B_r(x)\cap E)<\min\big\{\mm(B_\rho(y)\setminus E),\mm(B_\rho(\tilde y)\setminus E)\big\}
\quad\text{ for every }x\in K\text{ and }r\in(0,r_1).
\end{equation}
%Again by compactness of \(K\),
{
Thanks to \eqref{eq:R_x_unif_cpt},}
we can also find \(r_2>0\) such that \(r_2<R_y(E)\), \(r_2<R_{\tilde y}(E)\), and
{\(r_2<R_x(\varnothing)\) hold for every \(x\in K\)}. Now define \(\bar r_0\coloneqq\min\{r_0,r_1,r_2\}>0\).
Let \(x\in\partial^e E\cap B\) be fixed. Our choice of \(\rho\) ensures that \(\bar B_\rho(x)\) is disjoint from
at least one between \(\bar B_\rho(y)\) and \(\bar B_\rho(\tilde y)\). Up to relabeling \(y\) and \(\tilde y\), say that
\(\bar B_\rho(x)\cap\bar B_\rho(y)=\varnothing\). Given any \(r\in(0,\bar r_0)\), we deduce from \eqref{eq:density},
\eqref{eq:density2}, and the continuity of \(s\mapsto\mm(B_s(y)\setminus E)\) that there exists \(s(r)\in(\bar r_0,\rho)\)
such that \(\mm(B_r(x)\cap E)=\mm(B_{s(r)}(y)\setminus E)\). Define the Borel set \(E_r\subseteq\X\) as
\(E_r\coloneqq(E\setminus B_r(x))\cup B_{s(r)}(y)\). By the minimality assumption on \(E\), arguing as we did in the proof of Lemma \ref{lem:key_lemma} we obtain
\begin{equation}\label{eq:density3}
P(E,B_r(x))\leq\max\{C_z,C_{\tilde z}\}\frac{\mm(B_r(x)\cap E)}{\bar r_0}+P(B_r(x),E^{(1)}),
\end{equation}
for any $x\in\partial^e E\cap B$ and $\mathcal L^1\text{-a.e.\ }r\in(0,\bar r_0)$.
For any \(x\in\partial^e E\cap B\), define
\(A_x(E)\coloneqq\big\{r>0\,:\,|{\bf D}\sfd_x|(B_r(x)\cap E)\leq|{\bf D}\sfd_x|(B_r(x)\setminus E)\big\}\).
Fix \(\alpha>\max\{\log_2(C_D(\rho)),1\}\). Applying the relative isoperimetric inequality to the left-hand side
of \eqref{eq:density3} and using Lemma \ref{lem:var_dist}, we deduce that
\begin{equation}\label{eq:density4}\begin{split}
2 C_0\frac{\mm(B_r(x))^{\frac{1}{\alpha}}}{r}&|{\bf D}\sfd_x|(B_r(x)\cap E)^{1-\frac{1}{\alpha}}\\
&=2 C_0\frac{\mm(B_r(x))^{\frac{1}{\alpha}}}{r}\min\big\{|{\bf D}\sfd_x|(B_r(x)\cap E),|{\bf D}\sfd_x|(B_r(x)\setminus E)\big\}^{1-\frac{1}{\alpha}}\\
&\leq 2 C_0\frac{\mm(B_r(x))^{\frac{1}{\alpha}}}{r}\min\big\{\mm(B_r(x)\cap E),\mm(B_r(x)\setminus E)\big\}^{1-\frac{1}{\alpha}}\\
&\leq P(B_r(x),E^{(1)})+\frac{\max\{C_z,C_{\tilde z}\}}{c^{1-\frac{1}{\alpha}}}\frac{r}{\bar r_0}\frac{\mm(B_r(x))^{\frac{1}{\alpha}}}{r}
|{\bf D}\sfd_x|(B_r(x)\cap E)^{1-\frac{1}{\alpha}}
\end{split}\end{equation}
holds for \(\mathcal L^1\)-a.e.\ \(r\in(0,\bar r_0)\cap A_x(E)\), where we set
\(C_0\coloneqq 1/\big(2\,C_I(\alpha,\rho)^{(\alpha-1)/\alpha}\big)\) for brevity. Therefore, if we let
\[
\bar r\coloneqq\min\bigg\{\frac{c^{1-\frac{1}{\alpha}}C_0\bar r_0}{\max\{C_z,C_{\tilde z}\}},\bar r_0\bigg\}\in(0,\bar r_0],
\]
then we infer from \eqref{eq:density4} that
\begin{equation}\label{eq:density5}
C_0\frac{\mm(B_r(x))^{\frac{1}{\alpha}}}{r}|{\bf D}\sfd_x|(B_r(x)\cap E)^{1-\frac{1}{\alpha}}\leq P(B_r(x),E^{(1)})
\quad\text{ for }\mathcal L^1\text{-a.e.\ }r\in(0,\bar r)\cap A_x(E).
\end{equation}
This also proves (by considering \(\X\setminus E\) instead of \(E\)) that, up to shrinking \(\bar r>0\), it holds that
\begin{equation}\label{eq:density6}
C_0\frac{\mm(B_r(x))^{\frac{1}{\alpha}}}{r}|{\bf D}\sfd_x|(B_r(x)\setminus E)^{1-\frac{1}{\alpha}}\leq P(B_r(x),E^{(0)})
\quad\text{ for }\mathcal L^1\text{-a.e.\ }r\in(0,\bar r)\cap A_x(\X\setminus E).
\end{equation}
Let us now define the function \(f_x\colon(0,+\infty)\to\R\) as
\[
f_x(r)\coloneqq\min\big\{|{\bf D}\sfd_x|(B_r(x)\cap E^{(1)}),|{\bf D}\sfd_x|(B_r(x)\cap E^{(0)})\big\}\quad\text{ for every }r>0.
\]
Corollary \ref{cor:conseq_coarea} ensures that \(f_x\) is locally absolutely continuous and
\[
f'_x(r)=\left\{\begin{array}{ll}
P(B_r(x),E^{(1)})\\
P(B_r(x),E^{(0)})
\end{array}\quad\begin{array}{ll}
\text{ for }\mathcal L^1\text{-a.e.\ }r\in A_x(E),\\
\text{ for }\mathcal L^1\text{-a.e.\ }r\in A_x(\X\setminus E).
\end{array}\right.
\]
Observe that \(A_x(E)\cup A_x(\X\setminus E)=(0,+\infty)\). Arguing as in Lemma \ref{lem:key_lemma},
we deduce from \eqref{eq:density5} and \eqref{eq:density6} that
\begin{equation}\label{eq:density7}
\frac{C_0}{\alpha}\frac{\mm(B_r(x))^{\frac{1}{\alpha}}}{r}\leq(f_x^{\frac{1}{\alpha}})'(r)
\quad\text{ for }\mathcal L^1\text{-a.e.\ }r\in(0,\bar r).
\end{equation}
Given any \(r\in(0,\bar r)\), we can integrate the inequality in \eqref{eq:density7} over the interval \([r/2,r]\), thus obtaining that
\begin{equation}\label{eq:density8}\begin{split}
\frac{C_0}{2\alpha\,(C_D(\bar r/2))^{\frac{1}{\alpha}}}&\mm(B_r(x))^{\frac{1}{\alpha}} \leq\frac{C_0}{\alpha}\frac{\mm(B_{r/2}(x))^{\frac{1}{\alpha}}}{r}\frac{r}{2} \leq
\frac{C_0}{\alpha}\int_{r/2}^r\frac{\mm(B_s(x))^{\frac{1}{\alpha}}}{s}\,\d s\\&\leq\int_0^r(f_x^{\frac{1}{\alpha}})'(s)\,\d s=f_x(r)^\frac{1}{\alpha} \leq\min\big\{\mm(B_r(x)\cap E),\mm(B_r(x)\setminus E)\big\}^{\frac{1}{\alpha}}.
\end{split}\end{equation}
It follows that \(\mm(B_r(x))\leq C_1\mm(B_r(x)\cap E)\) for every \(x\in\partial^e E\cap B\) and \(r\in(0,\bar r)\),
where we define \(C_1\coloneqq C_D(\bar r/2)\big(\frac{2\alpha}{C_0}\big)^\alpha\).

Let \(x\in\partial^e E\cap B\) and \(r\in(0,\bar r)\) be fixed. Since
\(\mm(B_r(x))\leq C_1\min\big\{\mm(B_r(x)\cap E),\mm(B_r(x)\setminus E)\big\}\) by \eqref{eq:density8},
by using the relative isoperimetric inequality, and recalling that \(2C_0=1/C_I(\alpha,\rho)^{(\alpha-1)/\alpha}\), we get that
\[
\frac{2C_0}{C_1^{1-\frac{1}{\alpha}}}\frac{\mm(B_r(x))}{r}\leq
2C_0\frac{\mm(B_r(x))^{\frac{1}{\alpha}}}{r}\min\big\{\mm(B_r(x)\cap E),\mm(B_r(x)\setminus E)\big\}^{1-\frac{1}{\alpha}}\leq P(E,B_r(x)).
\]
On the other hand,
%since we can assume that \(C_2\coloneqq\sup_{x\in B}C_x(\varnothing)<+\infty\) by the compactness of \(K\),
{up to shrinking \(\bar r\) (depending only on \(B\)), we can find a constant \(C_2>0\) (depending only on \(B\)) such that
\(P(B_{\tilde r}(x))\leq C_2\frac{\mm(B_{\tilde r}(x))}{\tilde r}\) for every \(\tilde r\in(0,\bar r)\); recall the discussion after \eqref{eq:R_x_unif_cpt}. Then}
\[\begin{split}
P(E,B_r(x))&\leq P(E,B_{\tilde r}(x)) \leq\max\{C_z,C_{\tilde z}\}\frac{\mm(B_{\tilde r}(x)\cap E)}{\tilde r}+P(B_{\tilde r}(x),E^{(1)})
\\&\leq\max\{C_z,C_{\tilde z}\}\frac{\mm(B_{\tilde r}(x))}{\tilde r}+P(B_{\tilde r}(x)) \leq\big(\max\{C_z,C_{\tilde z}\}+C_2\big)\frac{\mm(B_{\tilde r}(x))}{\tilde r},
\end{split}\]
for \(\mathcal L^1\)-a.e.\ \(\tilde r\in(r,\bar r)\), thanks to \eqref{eq:density3} and to the deformation property.
Hence, \(\frac{r P(E,B_r(x))}{\mm(B_r(x))}\leq\max\{C_z,C_{\tilde z}\}+C_2\) for all \(x\in\partial^e E\cap B\) and
\(r\in(0,\bar r)\). Picking
\[
C\coloneqq\max\big\{C_1,C_1^{(\alpha-1)/\alpha}/(2C_0),\max\{C_z,C_{\tilde z}\}+C_2\big\},
\]
we conclude that \eqref{eq:density_claim} holds. Finally, applying \eqref{eq:density_claim} we conclude that for every
\(x\in\partial^e E\cap B\) and \(r\in(0,\bar r/2)\) it holds that
\[
\frac{P(E,B_{2r}(x))}{P(E,B_r(x))}\leq\frac{C\mm(B_{2r}(x))}{2r}\frac{Cr}{\mm(B_r(x))}\leq\frac{C^2 C_D(\bar r/2)}{2},
\]
which proves the validity of \eqref{eq:density_claim2}. Consequently, the statement is achieved.
\end{proof}
We conclude with a final comment on further minimality properties satisfied by volume-constrained minimizers. Such properties can be derived by reproducing well-known arguments, see, e.g. \cite[Remark 3.23, Theorem 3.24]{AntonelliPasqualettoPozzetta21}, exploiting Theorem \ref{thm:main} and the deformation property.
\begin{remark}\label{rem:LambdaQuasiMinimality}
Let \((\X,\sfd,\mm)\) be a length PI space having the deformation property. Let \(E\subseteq\X\) be a volume-constrained
minimizer of the perimeter. Using Theorem \ref{thm:main} and with arguments similar to those in the proof of Theorem
\ref{thm:density}, it is possible to prove that for any compact set $K\subseteq \X$ there exist $\Lambda, r_0>0$ such that
$E$ is a $(\Lambda, r_0)$-perimeter minimizer on $K$, i.e.\ whenever $F\Delta E \subseteq B_r(x)$ for some
$x \in K$ and $r<r_0$ it holds $P(E,B_r(x)) \le P(F,B_r(x)) + \Lambda\,\mm(E\Delta F)$.

Moreover, for any given compact set $K\subseteq \X$ there exist constants $L, r_0>0$ such that $E$ is $(L, r_0)$-quasi minimal
on $K$, i.e.\ whenever $F\Delta E \subseteq B_r(x)$ for some $x \in K$ and $r<r_0$ it holds that $P(E,B_r(x)) \le L\, P(F,B_r(x))$.
The class of quasi-minimal sets has been studied e.g.\ in \cite{KKLS13}.

It is worth pointing out that, once we know that volume-constrained minimizers of the perimeter are \((L,r_0)\)-quasi minimal sets,
Theorem \ref{thm:density} follows directly from \cite[Theorem 4.2 and Lemma 5.1]{KKLS13}. Nevertheless, we opted for a
self-contained proof of Theorem \ref{thm:density}, which takes advantage of the openness of volume-constrained minimizers.
\fr
\end{remark}
\section{Boundedness of isoperimetric sets}
In this last section, we prove the boundedness of isoperimetric sets in length PI spaces satisfying the deformation property
and with a uniform lower bound on the volume of unit balls  (Theorem \ref{thm:isoper_bdd}). The argument makes use of the
topological regularity given by our main result Theorem \ref{thm:main}.
\begin{proof}[Proof of Theorem \ref{thm:isoper_bdd}]
Suppose by contradiction that $E$ has no  bounded representatives, i.e.\ $\mea(E\setminus B_R(x))>0$ for all $R>0$ and $x \in \X.$
In particular $\X$ is unbounded and, since 
\[
v_0\coloneqq\inf_{x \in \X} \mea(B_1(x))>0,
\]
we have $\mea(\X)=\infty$ and
$\mea(\X\setminus E)>0$. Then $P(E)>0$ and, arguing as in the proof of Theorem \ref{thm:main}, we can find $y \in E^{(1)}$
and $\rho\in(0,R_y(E))$ such that $\delta\coloneqq \mea(B_\rho(y)\setminus E)>0$. By Theorem \ref{thm:main} it holds that
$y \in \Int(\eone)$, i.e.\ there exists $r_0>0$ such that $\mea(B_{r_0}(y)\setminus E)=0.$ { We consider the function $f: (0,+\infty) \to \mathbb R$ defined by
\[
f(R)\coloneqq |{\bf D} \sfd_y|(\eone \setminus  B_R(y))=|{\bf D} \sfd_y|(\eone)-|{\bf D} \sfd_y|(\eone\cap B_R(y)),
\]
and observe that $|{\bf D} \sfd_y|(\eone)\le \lip(\sfd_y) \mea(\eone)=\mea(E)<+\infty.$ By Corollary \ref{cor:conseq_coarea} the function $f$ is locally absolutely continuous and satisfies $f'(r)=-P(B_R(x),\eone)$ for $\mathcal L^1$-a.e.\ $R>0.$ 
Thanks to Lemma \ref{lem:var_dist} and since  \(\mm(E^{(1)}\Delta E)=0\), there also exists a constant $c>0$ such that 
\begin{equation}\label{eq:comparable md}
    0<c\,\mea (E\setminus B_R(y))\le f(R)\le \mea (E\setminus B_R(y)), \quad \forall\, R>0.
\end{equation}
}
Observe that, since $\mea(E)<+\infty$,
it holds $\mea(E\setminus B_R(y))\to 0$ as $R\to +\infty$. { Hence  $f(R)\to 0$ as  $R\to +\infty$
and so we can find $R_0>\rho$ such that $f(R)<\min\{\delta,v_0/2\}$ for all $R\ge R_0.$
}
% Moreover, by the coarea formula and recalling
% \eqref{eq:total variation and m}, we have that
% \[
% \int_{R}^{2R} P(B_t(y),\eone)\,\d t\le \int_{\rr} P(B_t(y),\eone\setminus B_{R}(y))\,\d t=
% |{\bf D}\sfd_y|(E^{(1)}\setminus B_R(y))\le \mea(E\setminus B_R(y)),
% \]
% for every  $R>0$, where in the first inequality we used the fact that the measure $P(B_t(y),\cdot)$ is concentrated on $\partial B_t(y)$.
% In particular, for all $R> 1$ the set 
% \[
% \big\{t \in (R,2R)\,:\,P(B_t(y),\eone)\le 2\,\mea(E\setminus B_R(y))\big\}
% \]
% has positive measure.
% Therefore, there exists a sequence $R_n\to +\infty$, with $R_n> \rho,$ such that 
% \begin{equation}
% \begin{split}
% P(B_{R_n}(y),\eone)\le 2\,&\mea(E\setminus B_{R_n}(y)),\quad 0<\mea( E \setminus B_{R_n}(y))<\min\{\delta,v_0/2\},\\
%  &P(E,\partial B_{R_n}(y))=0\quad\text{ for every }n \in \nn.
% \end{split}
% \end{equation}
By continuity, for every {$R\ge R_0$ there exists $r(R)\in(0,\rho)$ such that 
\begin{equation}\label{eq:equal n measure}
\mea(B_{r(R)}(y)\setminus E)=\mea( E \setminus B_{R}(y)).
\end{equation}}
For every {$R\ge R_0$ we define the set $F_R\coloneqq (E\cup B_{r(R)}(y))\cap B_{R}(y)$, which satisfies $\mea(F_R)=\mea(E)$
thanks to \eqref{eq:equal n measure} and $r(R)<R$}. Hence, by minimality, {$P(E)\le P(F_R)$ for every $R\ge R_0$.}
Moreover, using Proposition \ref{prop:formula_union} and the deformation property, for $\mathcal L^1$-a.e.\ $R\ge R_0$ we have
{
\begin{align*}
P(E)\le P(F_R)&=P((E\cup B_{r(R)}(y))\cap B_{R}(y))\\&\le P(E\cup B_{r(R)}(y),B_{R}(y)^{(1)})+P(B_{R}(y),(E\cup B_{r(R)}(y))^{(1)})\\
&\le P(E\cup B_{r(R)}(y))-P(E\cup B_{r(R)}(y),B_{R}(y)^{(0)})+P(B_{R}(y),\eone)\\
&\le P(E)+C_y(E)\frac{\mea(B_{r(R)}(y)\setminus E)}{r_0}-P(E,B_{R}(y)^{(0)})+P(B_{R}(y),\eone)\\
&\le P(E)+C_y(E)\frac{\mea(B_{r(R)}(y)\setminus E)}{r_0}-P(E\setminus B_{R}(y))+2P(B_{R}(y),\eone)\\
&\le P(E)+C_y(E)\frac{\mea(E\setminus B_{R}(y))}{r_0}-C v_0^\frac1\alpha\mea(E\setminus B_{R}(y))^{\frac{\alpha-1}{\alpha}}+2P(B_{R}(y),\eone),
\end{align*}
with $C>0, \alpha>1$ constants independent of $R,$}
where in the fifth line we used again Proposition \ref{prop:formula_union} and in the last line we used the isoperimetric
inequality for small volumes in Proposition \ref{thm:isop_small} (recall that $\mea( E \setminus B_{R}(y))<v_0/2$). {This combined with \eqref{eq:comparable md}  shows that 
\[
2f'(R)\le C_y(E)c^{-1}r_0^{-1} f(R)-C v_0^\frac1\alpha  f(R)^{\frac{\alpha-1}{\alpha}}\le -C_1f(R)^{\frac{\alpha-1}{\alpha}}, \quad \text{for a.e. } R\ge R_1,
\]
for some constant $R_1\ge R_0$ big enough  and where $C_1>0$ is a constant independent of $R$. Note that in the last inequality we used that $f(R)\to 0$ as $R\to +\infty$  and $\alpha>1.$  Since $f(R)>0$ for all $R>0$, this shows that
\[
(f^\frac{1}{\alpha})'(R)\le -\frac{C_1}{2\alpha }, \quad \text{for a.e. } R\ge R_1,
\]
which contradicts the fact that $f(R)$ is strictly positive for any $R>0$.
}
%  Hence, since $\mea(E\setminus B_{R}(y))\to 0$ and
% $\mea(E\setminus B_{R}(y))>0$, we obtain that for $n\in\nn$ big enough it holds that $P(E)<P(E),$ which is a contradiction.
\end{proof}
%
%
% \bigskip
% %
% \textbf{Statements and Declarations}

% \bigskip

% \textbf{Funding:} The third author is partially supported by the INdAM - GNAMPA Project 2022 CUP\_ E55F22000270001 ``Isoperimetric problems: variational and geometric aspects''.
% The fourth author is supported by the Academy of Finland project \emph{Incidences on Fractals}, Grant No.\ 321896.

% \textbf{Data Availibility Statement:} This research does not have any associated data.

% %
% %
% %
% \textbf{Competing interests:} The authors declare that they have no conflict of interest.
%

%
%
%
%
\end{document}